\documentclass[a4paper,reqno]{amsart}

\usepackage{amsthm}
\usepackage{amsmath}
\usepackage{amssymb}
\usepackage{centernot}   
\usepackage{graphicx}        
\usepackage{multicol}        
\usepackage[all,ps]{xy}
\usepackage{layout}
\usepackage{booktabs}

\newcommand{\CC}{{\mathbb C}}

\newcommand{\cB}{{\mathcal B}}

\newcommand{\cI}{{\mathcal I}}

\newcommand{\cL}{{\mathcal L}}

\newcommand{\cO}{{\mathcal O}}

\newcommand{\cU}{{\mathcal U}}

\newcommand{\cX}{{\mathcal X}}

\newcommand{\es}{\emptyset}

\newcommand{\FF}{{\mathbb F}}

\newcommand{\hra}{\hookrightarrow}
\newcommand{\la}{\langle}
\newcommand\lagr{\mathbb{LG}(\bigwedge^ 3 \CC^6)}

\newcommand{\lra}{\longrightarrow}

\newcommand{\n}{\noindent}

\newcommand{\ov}{\overline}
\newcommand{\PP}{{\mathbb P}}

\newcommand{\ra}{\rangle}

\newcommand{\wt}{\widetilde}

\newcommand{\Gr}{\mathrm{Gr}}

\theoremstyle{plain}

\newtheorem{thm}{Theorem}[section]

\newtheorem{clm}[thm]{Claim}

\newtheorem{crl}[thm]{Corollary}

\newtheorem{lmm}[thm]{Lemma}
\newtheorem{prp}[thm]{Proposition}
\newtheorem{prp-dfn}[thm]{Proposition-Definition}

\theoremstyle{definition}

\newtheorem{dfn}[thm]{Definition}

\theoremstyle{remark}

\newtheorem{rmk}[thm]{Remark}

\DeclareMathOperator{\im}{Im}

\DeclareMathOperator{\mult}{mult}

\DeclareMathOperator{\sing}{sing}

\DeclareMathOperator{\Sym}{S}

\DeclareMathOperator{\vol}{vol}

\usepackage{ifthen}
\usepackage{rotating}
\usepackage{supertabular}
\newcommand{\cit}[1]{{\rm \textbf{#1}}}
\newcommand{\Ref}[2]{\cit{%
\ifthenelse{\equal{#1}{thm}}{Theorem}{}%
\ifthenelse{\equal{#1}{ass}}{Assumption}{}%
\ifthenelse{\equal{#1}{chp}}{Chapter}{}%
\ifthenelse{\equal{#1}{prp}}{Proposition}{}%
\ifthenelse{\equal{#1}{lmm}}{Lemma}{}%
\ifthenelse{\equal{#1}{cnj}}{Conjecture}{}%
\ifthenelse{\equal{#1}{crl}}{Corollary}{}%
\ifthenelse{\equal{#1}{dfn}}{Definition}{}%
\ifthenelse{\equal{#1}{expl}}{Example}{}%
\ifthenelse{\equal{#1}{hyp}}{Hypothesis}{}%
\ifthenelse{\equal{#1}{rmk}}{Remark}{}%
\ifthenelse{\equal{#1}{clm}}{Claim}{}%
\ifthenelse{\equal{#1}{exe}}{Exercise}{}%
\ifthenelse{\equal{#1}{sec}}{Section}{}%
\ifthenelse{\equal{#1}{subsec}}{Subsection}{}%
\ifthenelse{\equal{#1}{univ}}{Universal Property}{}%
\ifthenelse{\equal{#1}{trm}}{Terminology}{}%
\ifthenelse{\equal{#1}{tbl}}{Table}{}%
\  \ref{#1:#2}%
}}

\usepackage{multirow}

 \makeindex
 
\begin{document}
 \title{Pairwise incident planes and Hyperk\"ahler four-folds}
 \author{Kieran G. O'Grady\\\\
\lq\lq Sapienza\rq\rq Universit\`a di Roma}
\date{April 12 2012}
\thanks{Supported by
 PRIN 2010}
 \dedicatory{Dedicato a Joe in occasione del $60^0$ compleanno}
  \maketitle
 \tableofcontents
 \section{Introduction}\label{sec:prologo}
 \setcounter{equation}{0}
A family of pairwise incident lines in a projective space consists of lines through a point or lines contained in a plane. Is there an analogous characterization of families of pairwise incident planes in a complex projective space ?  A beautiful  theorem of Ugo Morin~\cite{morin} states that an algebraic {\bf irreducible} family of pairwise incident planes is contained in one of the following families:
\begin{enumerate}
\item[(1)]
Planes containing a fixed point.
\item[(2)]
Planes whose intersection with a fixed plane has dimension at least $1$.
\item[(3)]
Planes contained in a fixed $4$-dimensional projective space.
\item[(4)]
One of the two irreducible components of the set of planes contained in a fixed smooth $4$-dimensional quadric.
\item[(5)]
The planes tangent to a fixed Veronese surface (image of $\PP^2\to |\cI_{\PP^2}(2)|^{\vee}$).
\item[(6)]
The planes intersecting a fixed Veronese surface along a conic.
\end{enumerate}
In the present paper we will address the following question: what are the  cardinalities
 of finite families of pairwise incident planes ? As stated the question is not interesting because the families  of pairwise incident planes listed above contain sets of arbitrary finite cardinality.
 In order to formulate a meaningful question we recall the following definition of Morin: a family of pairwise incident planes is   {\it complete}  if  there exists no plane outside the family which is incident to all planes in the family - in other words if the family is maximal. We ask the following question:  what are  the cardinalities of  finite complete family of pairwise incident planes ?  Before stating our main result we will describe a finite complete family of pairwise incident planes in $\PP^6$. Let $\{v_0,\ldots,v_6\}$ be a basis of $\CC^7$. Let $\Lambda_1,\ldots,\Lambda_7\in\Gr(2,\PP^7)$ be defined by
\begin{equation}\label{primiquattro}
\scriptstyle
\Lambda_1=\PP\la v_0,v_1,v_2\ra,\quad 
\Lambda_2=\PP\la v_2,v_3,v_4\ra,\quad 
\Lambda_3=\PP\la v_0,v_4,v_5\ra,\quad  
\Lambda_4=\PP\la v_1,  v_3, v_5, \ra,
\end{equation}
\begin{equation}\label{altritre}
\scriptstyle
\Lambda_5=\PP\la v_0,v_3,v_6\ra,\quad  
\Lambda_6=\PP\la v_1,v_4,v_6\ra,\quad 
\Lambda_7=\PP\la v_2,v_5,v_6\ra. 
\end{equation}
As is easily checked the planes $\Lambda_1,\ldots,\Lambda_7$ are pairwise incident: we will show (see~\Ref{clm}{primacompl}) that they form a complete family. 
\begin{rmk}
We identify the set $\{[v_0],\ldots,[v_6]\}$ and  $\PP^2_{\FF_2}$ (the projective plane on the field with $2$ elements) as follows:
\begin{equation*}
\scriptstyle
[v_0]\mapsto [010],\ \ [v_1]\mapsto [011],\ \ [v_2]\mapsto [001],\ \ [v_3]\mapsto [101],\ \ [v_4]\mapsto [100],
\ \ [v_5]\mapsto [110],\ \ [v_6]\mapsto [111]. 
\end{equation*}
Given the above identification a plane in $\PP^6$ is equal to one of 
 the $\Lambda_i$'s if and only if it  is   spanned by the points of a line in  $\PP^2_{\FF_2}$.
\end{rmk}
\begin{thm}\label{thm:alpiuventi}
Let $T\subset\Gr(2,\PP^N)$ be a  finite complete family of pairwise incident planes. The planes in $T$ span a projective space of dimension $5$ or $6$.  If the span has dimension $6$ then $T$ is projectively equivalent to the family $\{\Lambda_1,\ldots,\Lambda_7\}$ described above.   If the span has dimension $5$ then $T$ has at most $20$ elements. For any  $10\le k\le 16$ there  exists a  complete family of $k$ pairwise incident planes: in fact  it has at least $(20-k)$ moduli.  
\end{thm}
In~\Ref{sec}{dimambiente} we will study finite complete families of pairwise incident planes which span a projective space of dimension greater than $5$: the proofs are of an elementary  nature. In~\Ref{sec}{pianilagr} we will make the connection between our question and the geometry of certain Hyperk\"ahler $4$-folds which are double covers of special sextic hypersurfaces in $\PP^5$ named EPW-sextics. Then we will 
apply results of Ferretti~\cite{ferretti} on degenerations of double EPW-sextics  in order to show that there exist  finite complete families of pairwise incident planes in $\PP^5$ of cardinality between $10$ and $16$, we will also get the lower bound on the number of moduli given in~\Ref{thm}{alpiuventi}.  In~\Ref{sec}{stimasup} we will prove that  a  finite complete family of pairwise incident planes has cardinality at most $20$. 

A few comments. I suspect that $16$ is the  maximum cardinality of a finite complete family of pairwise incident planes.  Our (we might say Ferretti's) proof that there exist complete families of pairwise incident planes of cardinality between $10$ and $16$  is a purely existential proof: it does not give explicit families.   One may ask for explicit examples. The paper~\cite{dolgamark} of  Dolgachev and Markushevich provides a general framework for the study of this problem. In particular the authors
 associate to a generic Fano model of an Enriques surface (plus a suitable choice of $10$ elliptic curves on  the surface) a finite collection of complete families of $10$ pairwise incident planes in $\PP^5$ - they also study the problem of classifying the irreducible components (there are several such) of the locus parametrizing ordered $10$-tuples of pairwise incident planes in $\PP^5$.
 In the same paper Dolgachev and Markushevich have given  explicit constructions of    complete families of  $13$ pairwise incident planes.

\vskip 2mm
\n
{\bf Notation and conventions.}
We work throughout over $\CC$. 
Let $T\subset\Gr(2,\PP^N)$ be a  family of  planes: the {\it span} of $T$ is the span of the union of the planes parametrized by $T$. 
 \section{Families of pairwise incident planes in $\PP^N$ for $N>5$}\label{sec:dimambiente}
 \setcounter{equation}{0}
Let $T\subset\Gr(2,\PP^N)$ be a  finite complete family of pairwise incident planes. If  the span of $T$  is contained in a projective space $M$ of dimension at most $4$ then $T$ is contained in the infinite family of pairwise incident planes $\Gr(2,M)$, that is a contradiction. Hence the span of $T$ has dimension at least $5$.
In the present section we will classify finite complete family of pairwise incident planes whose span has dimension greater than $5$. 
 We will start  by showing that the  planes $\Lambda_1,\ldots,\Lambda_7\subset\PP^7$  defined by~\eqref{primiquattro}, \eqref{altritre} form a complete family of pairwise incident planes. 
 Let $v_0,\ldots,v_6$ be as in~\Ref{sec}{prologo}; we let
\begin{equation}\label{pokemon}
\PP^5:=\PP\la v_0,\ldots,v_5\ra.
\end{equation}
The set of lines in $\PP^5$ meeting $\Lambda_1,\Lambda_2,\Lambda_3$ has $4$ irreducible components, each isomorphic to $\PP^2$; more precisely 
\begin{multline}\label{intertre}
\scriptstyle
\{L\in\Gr(1,\PP^5) \mid L\cap \Lambda_i\not=\es,\ i=1,2,3\}= 
\Gr(1,\PP\la v_0,v_2,v_4\ra) \cup \\
\scriptstyle
\cup \{\PP\la v_0,u\ra \mid 0\not=u\in \la v_2,v_3,v_4\ra\}\cup
\{\PP\la v_2,u\ra \mid 0\not=u\in \la v_0,v_4,v_5\ra\} 
\cup \{\PP\la v_4,u\ra \mid 0\not=u\in \la v_0,v_1,v_2\ra\}.
\end{multline}
From the above equality one  gets that there are exactly $3$ lines in $\PP^5$  meeting $\Lambda_1,\ldots,\Lambda_4$. More precisely let 
\begin{equation}\label{unduetre}
\scriptstyle
L_5:=\PP\la v_0,v_3\ra=\Lambda_5\cap\PP^5,\quad L_6:=\PP\la v_1,v_4\ra=\Lambda_6\cap\PP^5,\quad  
L_7:=\PP\la v_2,v_5\ra=\Lambda_7\cap\PP^5.
\end{equation}
Then
\begin{equation}\label{trerette}
\{L\in\Gr(1,\PP^5) \mid L\cap \Lambda_i\not=\es,\, i=1,2,3,4\}= 
\{L_5, L_6, L_7 \}.
\end{equation}
\begin{clm}\label{clm:primacompl}
The collection of planes $\Lambda_1,\ldots,\Lambda_7\subset\PP^6$  defined by~\eqref{primiquattro}, \eqref{altritre} is a complete family of pairwise incident planes. 
\end{clm}
\begin{proof}
We need to show that the family is  complete. First we notice that the span of $\Lambda_1,\ldots,\Lambda_4$ is equal to  $\PP^5$, notation as in~\eqref{pokemon}.
Now let $\Lambda\subset \PP^6$ be a plane intersecting $\Lambda_1,\ldots,\Lambda_7$. Since the intersection of $\Lambda_1,\ldots,\Lambda_4$ is empty one of the following holds:
\begin{enumerate}
\item[(1)]
$\Lambda\subset \PP^5$,
\item[(2)]
$\dim(\Lambda\cap \PP^5)=1$.
\end{enumerate}
Suppose that~(1) holds. Then $\Lambda$ meets each of the lines $L_5,L_6,L_7$ given by~\eqref{unduetre}.  Since $L_5,L_6,L_7$  generate $\PP^5$ it follows that $\Lambda$ intersects $L_i$  in a single point $p_i$ and that $\Lambda$ is spanned by $p_5,p_6,p_7$. Imposing the condition that $\la p_5,p_6,p_7\ra$ (for $p_i\in L_i$) meet each of $\Lambda_1,\ldots,\Lambda_4$ we get that  $\la p_5,p_6,p_7\ra$ is one of  $\Lambda_1,\ldots,\Lambda_4$. This proves that if~(1) holds then $\Lambda\in\{\Lambda_1,\ldots,\Lambda_4\}$. Next suppose that~(2) holds and let $L= \Lambda\cap \PP^5$. Then $L$ meets each of $\Lambda_1,\ldots,\Lambda_4$. By~\eqref{trerette} it follows that $L$ equals one of $L_5,L_6,L_7$. Suppose that  $L=L_5$. Then $\Lambda$ meets $\Lambda_6$ and $\Lambda_7$ in points outside $\PP^5$. Now notice that the span of  $\Lambda,\Lambda_6,\Lambda_7$ is all of $\PP^6$: it follows that $\Lambda,\Lambda_6,\Lambda_7$ meet in a single point, which is necessarily  $[v_6]$. Thus $\Lambda=\Lambda_5$. If $L$ equals one of $L_6$ or $L_7$ a similar argument shows that  $\Lambda=\Lambda_6$ or  $\Lambda=\Lambda_7$ respectively.
\end{proof}
Our next goal is to  prove that if $T$ is  a finite complete family of pairwise incident planes spanning a projective space of dimension greater than $5$ then  $T$ is projectively equivalent to  $\{\Lambda_1,\ldots,\Lambda_7\}$ where 
the  planes $\Lambda_1,\ldots,\Lambda_7$  are defined by~\eqref{primiquattro}, \eqref{altritre}. First we make the following observation.
\begin{prp}\label{prp:unpunto}
Let $T\subset\Gr(2,\PP^N)$ be a family of pairwise incident planes. Suppose that there exist $\Lambda,\Lambda'\in T$ such   that their intersection is a line. Then $T$ is contained in an infinite family of pairwise incident planes.
\end{prp}
\begin{proof}
Let $L:=\Lambda\cap\Lambda'$ and $M:=\la \Lambda,\Lambda'\ra$. Thus $L$ is a line and $M$ is a $3$-dimensional projective space. Let $\Lambda''\in T$: since $\Lambda''$ intersects both $\Lambda$ and $\Lambda'$ one of the following holds:
\begin{enumerate}
\item[(1)]
$\dim(\Lambda''\cap M)\ge 1$,
\item[(2)]
$\Lambda''\cap L\not=\es$.
\end{enumerate}
Now let $\Lambda_0\subset M$ be a plane containing $L$. If~(1) holds then $\Lambda_0$ intersects $(\Lambda''\cap M)$, if~(2) holds then $\Lambda_0$ contains the non-empty intersection $(\Lambda''\cap L)$: in both cases we get that  $\Lambda_0$  intersects $\Lambda''$.  Hence the union of $T$ and the set of  planes in $M$ containing $L$  is an infinite family of  pairwise incident  planes containing $T$. 
\end{proof}
The result below follows immediately from~\Ref{prp}{unpunto}.
\begin{crl}\label{crl:unpunto}
Let $T\subset\Gr(2,\PP^N)$ be a finite complete family of pairwise incident planes. If $\Lambda,\Lambda'\in T$ are distinct their intersection is a single point.
\end{crl}
\begin{prp}
Let $T\subset\Gr(2,\PP^N)$ be a finite complete family of pairwise incident planes. Suppose that the span of $T$ has dimension greater than $5$. Then $T$ is projectively equivalent to $\{\Lambda_1,\ldots,\Lambda_7\}$ where $\Lambda_1,\ldots,\Lambda_7$ are as in~\eqref{primiquattro} and~\eqref{altritre}. 
\end{prp}
\begin{proof}
Let $\Lambda_1,\Lambda_2\in T$ be distinct: by~\Ref{crl}{unpunto} they intersect in a single point $p$ and hence they span a $4$-dimensional projective space $M$. We claim that there does exist $\Lambda_3\in T$ which is not contained in $M$ and which intersects  $\Lambda_1,\Lambda_2$ in distinct points. In fact suppose the contrary. Then we get an infinite family of pairwise incident planes by adding to $T$ the planes $\Lambda\in\Gr(2,M)$ containing $p$: that contradicts the hypothesis that $T$ is a finite complete family of pairwise incident planes. Since the planes $\Lambda_1,\Lambda_2,\Lambda_3$ have distinct pairwise intersections  and they span a $5$-dimensional projective space 
there exists linearly independent $v_0,\ldots,v_5\in\CC^6$ such that $\Lambda_1,\Lambda_2,\Lambda_3$ are as in~\eqref{primiquattro}.  Now let $\Lambda\in T$: since $\Lambda$ intersects  $\Lambda_1,\Lambda_2,\Lambda_3$ one of the following holds:
\begin{enumerate}
\item[(1)]
$\Lambda\subset\PP^5$ (notation as in~\eqref{pokemon}). 
\item[(2)]
$\dim(\Lambda\cap \PP^5)=1$ and the line $\Lambda\cap \PP^5$ is one of $L_5,L_6,L_7$, see~\eqref{unduetre}.
\end{enumerate}
We claim that there exists $\Lambda_4\in T$ which is contained in $\PP^5$ and does not intersect $\PP\la v_0,v_2,v_4\ra$.
In fact if no such $\Lambda_4$ exists then the plane $\PP\la v_0,v_2,v_4\ra$ is incident to all planes in $T$ and intersects each of $\Lambda_1,\Lambda_2,\Lambda_3$ along a line: that is a contradiction because of~\Ref{prp}{unpunto}. We may rename $v_1,v_3,v_5$ so that $\Lambda_4$ is as in~\eqref{primiquattro}. Now  notice that since the span of $T$ has dimension greater than $5$ there does exist $\Lambda\in T$ such that Item~(2) holds.  By~\Ref{crl}{unpunto} we have an injection
\begin{equation}\label{intermappa}
\begin{matrix}
T\setminus\Gr(2,\PP^5) & \hra & \{L_5,L_6,L_7\} \\
\Lambda & \mapsto & \Lambda\cap\PP^5
\end{matrix}
\end{equation}
 We claim that  Map~\eqref{intermappa} is surjective. In fact suppose that the image consists of a single line $L_i$: then every plane containing $L_i$ is incident to every plane in $T$, that contradicts the hypothesis that $T$ is a finite complete family of pairwise incident planes. Now suppose that  the image consists of $2$ lines: without loss of generality we may assume that they are $L_5,L_6$.
A straightforward computation gives that 
\begin{multline}\label{tittidisegna}
\scriptstyle
\{\Lambda\in\Gr(2,\PP^5) \mid \text{$\Lambda$ is incident to $L_5$, $L_6$, $\Lambda_1$, $\Lambda_2$, $\Lambda_3$ and $\Lambda_4$}\}=\PP\la v_0,v_1,v_3,v_4\ra
\cup \{\PP\la v_0,v_1,av_2+bv_3+cv_4\ra\}\cup \\
\scriptstyle
\cup \{\PP\la v_1,v_3,av_0+bv_4+cv_5\ra\}
\cup \{\PP\la v_3,v_4,av_0+bv_1+cv_2\ra\}
\cup \{\PP\la v_0,v_4,av_1+bv_3+cv_5\ra\}.
\end{multline}
Now notice that the right-hand side of~\eqref{tittidisegna} is an infinite family of pairwise incident planes: that contradicts the hypothesis that $T$ is a finite complete family of pairwise incident planes. We have proved  that  
Map~\eqref{intermappa} is surjective. Now let $\Lambda\in T$ be such that Item~(1) holds: then $\Lambda$ is incident to $\Lambda_1,\ldots,\Lambda_4$ and to $L_1,L_2,L_3$: it follows  that $\Lambda\in\{\Lambda_1,\ldots,\Lambda_4\}$ - see the proof of~\Ref{clm}{primacompl}. The set of  $\Lambda\in T$  such that Item~(2) holds consists of $3$ elements, say $\{\Lambda_5, \Lambda_6, \Lambda_7 \}$ where $\Lambda_i\cap \PP^5=L_i$. Since  $L_5,L_6,L_7$ span $\PP^5$ the planes 
$\Lambda_5, \Lambda_6, \Lambda_7 $ intersect in a single point which lies outside $\PP^5$: thus we may complete $v_0,\ldots,v_5$ to a basis of $\CC^7$ by adding a vector $v_6$ such that  $\Lambda_5\cap\Lambda_6\cap\Lambda_7 =\{[v_6]\}$. Then it is clear that $T$ is projectively equivalent  to $\{\Lambda_1,\ldots,\Lambda_7\}$.
\end{proof}
 \section{Complete finite families of pairwise incident planes in $\PP^5$}\label{sec:pianilagr}
 \setcounter{equation}{0}
In the present section we will associate to a finite complete family of pairwise incident planes in $\PP^5$ an EPW-sextic -  a special sextic hypersurface in $\PP^5$  which  comes equipped with a double cover. The double cover of a generic EPW-sextic  is a Hyperk\"ahler $4$-fold deformation equivalent to the Hilbert square of a $K3$. There is a divisor $\Sigma$  in the space of EPW-sextics whose generic point corresponds  to a double cover $X$ whose singular locus is  a  $K3$-surface of degree $2$: it is obtained from a HK $4$-fold $\wt{X}$ by contracting a divisor $E$ which is a conic bundle over the $K3$, see~\cite{ogper}.  Let $Y$ be the EPW-sextic corresponding to $X$: the covering map $X\to Y$ takes the singular locus of $X$ to a plane. There are more special EPW-sextics parametrized by points of $\Sigma$ which correspond to a HK $4$-fold $\wt{X}$ containing more than one of the divisors $E$: the images of these divisors under the composition $\wt{X}\to X\to Y$ are pairwise incident planes. 
 We will show that certain of these EPW-sextics  (introduced by Ferretti~\cite{ferretti}) provide examples of complete families of $k$ pairwise incident planes in $\PP^5$ for $10\le k\le 16$.
  Choose a volume-form  
$\vol\colon\bigwedge^6 \CC^6\overset{\sim}{\lra}\CC$
 and  equip $\bigwedge^3 \CC^6$ with the symplectic form
\begin{equation}
  (\alpha,\beta):=\vol(\alpha\wedge\beta).
\end{equation}
Let $A\subset\bigwedge^3\CC^6$ be a subspace: we let
\begin{eqnarray}\label{eccoteta}
\Theta_A & := & \{W\in\Gr(3,\CC^6) \mid \bigwedge^3 W\subset A\},\\
{\bf \Theta}_A & :=& \{\Lambda\in\Gr(2,\PP^5) \mid \text{$\Lambda=\PP(W)$ where $W\in \Theta_A$}\}.
\end{eqnarray}
The following simple observation will be our  starting point.
\begin{rmk}\label{rmk:incidiso}
Let  $A\subset\bigwedge^3\CC^6$ be {\bf isotropic} for the symplectic form $(,)$. Then ${\bf \Theta}_A$ is a family of pairwise incident planes.  Conversely let $T\subset\Gr(2,\PP^5)$ be a family of pairwise incident planes and $B\subset\bigwedge^3\CC^6$ be the subspace spanned by the vectors $\bigwedge^3 W$ for $W\in\Gr(3,\CC^6)$  such that $\PP(W)\in T$: then $B$ is isotropic for $(,)$.
\end{rmk}
Let $\lagr$ be the symplectic Grassmannian parametrizing Lagrangian subspaces of $\bigwedge^3 \CC^6$ - of course $\lagr$ does not depend on the choice of volume-form. Notice that $\dim\bigwedge^3 \CC^6=20$ and hence elements of $\lagr$ have dimension $10$.
\begin{clm}\label{clm:incilagr}
Let $T\subset\Gr(2,\PP^5)$ be a {\bf complete} family of pairwise incident planes. Then there exists  $A\in\lagr$ such that 
\begin{equation}\label{uguali}
{\bf \Theta}_A=T.
\end{equation}
Conversely suppose that
 $A\in\lagr$  is {\bf spanned} by $\Theta_A$ (embedded in $\bigwedge^3\CC^6$ by Pl\"ucker). Then ${\bf \Theta}_A $ is a complete family of pairwise incident planes. 
\end{clm}
\begin{proof}
Let $B\subset\bigwedge^3\CC^6$ be the subspace spanned by the vectors $\bigwedge^3 W$ for $W\in\Gr(3,\CC^6)$  such that $\PP(W)\in T$: then $B$ is $(,)$-isotropic, see~\Ref{rmk}{incidiso}.  Thus there exists $A\in\lagr$ containing $B$. Then ${\bf \Theta}_A$ is a family of pairwise incident planes, see~\Ref{rmk}{incidiso}, and it contains $T$. Since $T$ is complete we get that~\eqref{uguali} holds.  Now suppose that
 $A\in\lagr$  is spanned by $\bigwedge^3 W_1,\ldots,\bigwedge^3 W_{10}$  where $W_1,\ldots,W_{10}\in\Theta_A$.  
 Suppose that $\PP(W_{*})\in\Gr(2,\PP^5)$ is incident to all $\Lambda\in {\bf \Theta}_A$. Then $\bigwedge^3 W_{*}$ is orthogonal  to  
 $\bigwedge^3 W_1,\ldots,\bigwedge^3 W_{10}$ and hence to all of $A$. Since $A$ is lagrangian we get that   $\PP(W_{*})\in{\bf \Theta}_A$.  This proves that ${\bf \Theta}_A $ is a complete family of pairwise incident planes. 
\end{proof}
Let $A\in\lagr$: according to Eisenbud-Popescu-Walter (see the appendix of~\cite{eispopwal} or~\cite{ogepw}) one associates to $A$ a subset of $\PP^5$ as follows. Given a non-zero $v\in \CC^6$ we let 
\begin{equation}
F_v:=\{\alpha\in\bigwedge^3 \CC^6\mid v\wedge\alpha=0\}.
\end{equation}
 Notice that  $F_v\in \lagr$. We let
\begin{equation}
Y_A=\{[v]\in\PP^5\mid F_v\cap A\not=\{0\}\}.
\end{equation}
The lagrangians $F_v$ are the fibers of a vector-bundle $F$ on $\PP^5$ with $\det F\cong\cO_{\PP^5}(-6)$: it follows that $Y_A$ is the zero-locus of a section of $\cO_{\PP^5}(6)$. Thus either $Y_A=\PP^5$ (this happens for \lq\lq degenerate\rq\rq\ choices of $A$, for example  $A=F_w$) or else $Y_A$ is a sextic hypersurface - an {\it EPW-sextic}. We emphasize that EPW-sextics are very special hypersurfaces, in particular their singular locus has dimension at least $2$.  
An EPW-sextic $Y_A$ comes equipped with a finite map~\cite{ogdoppio}
\begin{equation}\label{rivest}
f_A\colon X_A\to Y_A.
\end{equation}
$X_A$ is the {\it double EPW-sextic} associated to $A$. The following result~\cite{ogepw} motivates the adjective \lq\lq double\rq\rq. Suppose that  
\begin{equation}\label{belcaso}
\text{$\Theta_A=\es$ and $\dim(F_v\cap A)\le 2$ for all $[v]\in\PP^5$.}
\end{equation}
(A dimension count shows that~\eqref{belcaso} holds for generic $A\in\lagr$.)
Then $Y_A\not=\PP^5$ and $X_A$ is a Hyperk\"ahler variety deformation equivalent to the Hilbert square of a $K3$ surface\footnote{Notice that if $A$ is general then $X_A$   is not isomorphic nor birational to the Hilbert square of a $K3$.}, moreover~\eqref{rivest} is identified with the quotient map of an anti-symplectic involution on $X_A$. What if one of the conditions of~\eqref{belcaso} are violated ? If $\Theta_A$ is empty  but there do exist $[v]\in\PP^5$ such that $\dim(F_v\cap A)> 2$ then necessarily $\dim(F_v\cap A)=3$ and $X_A$ is obtained from a holomorphic symplectic $4$-fold by contracting certain copies of $\PP^2$ (one for each point violating the second condition of~\eqref{belcaso}): thus $X_A$ is almost as good as a HK variety. On the other hand suppose that  $\Lambda\in{\bf \Theta}_A$: then $\Lambda\in Y_A$ and $Y_A$   
 and $X_A$ (assuming that $Y_A\not=\PP^5$) may be quite singular along $\Lambda$. 
The following result will  be handy.
\begin{prp}[Cor.~2.5 of~\cite{ogtasso} and Prop.~1.11, Claim 1.12 of~\cite{ogdoppio}]\label{prp:neve}
Let $A\in\lagr$ and $[v]\in\PP^5$.    Then the following hold:
\begin{itemize}
\item[(1)]
If no $\Lambda\in{\bf\Theta}_A$ contains $[v]$ then $Y_A\not=\PP^5$,  $\mult _{[v_0]}Y_A=\dim(A\cap F_{v_0})$ and 
\begin{enumerate}
\item[(1a)]
 if $\dim(F_v\cap A)\le 2$ then $X_A$ is smooth at $f_A^{-1}([v])$,
\item[(1b)]
 if $\dim(F_v\cap A)> 2$ then the analytic germ of $X_A$ at $f_A^{-1}([v])$ (a single point) is isomorphic to the cone over $\PP(\Omega^1_{\PP^2})$.
\end{enumerate}
\item[(2)]
If there exists $\Lambda\in{\bf\Theta}_A$ containing $[v]$ then either $Y_A=\PP^5$ or else $X_A$ is singular at $f_A^{-1}([v])$. 
\end{itemize}
\end{prp}
Next we will define an $A\in\lagr$ such that $Y_A$ is a triple quadric: the example will be a key element in the construction of   complete families of pairwise incident planes of cardinality between $10$ and $16$. Choose an isomorphism $\CC^6=\bigwedge^2 U$ where $U$ is a complex vector-space of dimension $4$. Thus $\Gr(2,U)\subset\PP(\CC^6)$ is a smooth quadric hypersurface: we let
\begin{equation}
Q(U):=\Gr(2,U).
\end{equation}
 We have an embedding
\begin{equation}\label{mappapiu}
\begin{matrix}
\PP(U) & \overset{i_{+}}{\hra} & \Gr(2,\PP^5)\\
[u_0] & \mapsto & \PP\{u_0\wedge u\mid u\in U\}
\end{matrix}
\end{equation}
\begin{dfn}
Let $A_{+}(U)\subset\bigwedge^3(\CC^6)$ be the 
subspace spanned by the cone over $\im(i_{+})$ - here we view $\Gr(2,\PP^5)$ as embedded in $\PP(\wedge^3\CC^6)$ by the Pl\"ucker map. 
\end{dfn}
Let $\cL$ be  Pl\"ucker line-bundle on $\Gr(2,\PP^5)$. Then $i_{+}^{*}\cL\cong \cO_{\PP(U)}(2)$ and the induced map on global sections is surjective: thus $\dim A_{+}(U)=10$. On the other hand any two planes in the image of $i_{+}$ are incident: thus  $A_{+}(U)\in\lagr$, see~\Ref{rmk}{incidiso}.  One has (see Claim 2.14 of~\cite{ogtasso}) 
\begin{equation}
Y_{A_{+}(U)}=3 Q(U).
\end{equation}
Let ${\bf K}\subset \PP(U)$ be a Kummer quartic surface and let ${\bf p}_1,\ldots,{\bf p}_{16}$ be its nodes. Choose  $k$  nodes ${\bf p}_{i_1},\ldots,{\bf p}_{i_k}$. There exist arbitrarily small deformations of ${\bf K}$ which contain exactly $k$ nodes which are small deformations of ${\bf p}_{i_1},\ldots,{\bf p}_{i_k}$ and are smooth elsewhere (it suffices to deform the minimal desingularization of ${\bf K}$ keeping  the rational curves lying over ${\bf p}_{i_1},\ldots,{\bf p}_{i_k}$ of type $(1,1)$ and not keeping of type $(1,1)$  the rational curves lying over the remaining nodes). Let $S_0$ be such a small deformation of ${\bf K}$  and $p_1,\ldots,p_{k}$ be its nodes. 
Let $\wt{S}_0\to S_0$ be the minimal desingularization: thus $\wt{S}_0$ is a $K3$ surface containing $k$ smooth rational curves $R_1,\ldots,R_{k}$ mapping to $p_1,\ldots,p_{k}$ respectively. The HK $4$-fold $\wt{S}_0^{[2]}$ contains $k$ disjoint copies of $\PP^2$ namely $R^{(2)}_1,\ldots,R^{(2)}_{k}$. We have a regular map
\begin{equation}\label{corda}
\begin{matrix}
\wt{S}_0^{[2]}\setminus \bigcup_{i=1}^{k}R^{(2)}_i & \lra & Q(U) \\
& & \\
Z & \mapsto & \la Z\ra 
\end{matrix}
\end{equation}
where $\la Z\ra$ is the unique line containing the scheme $Z$. One cannot extend the above map to a regular map over $R^{(2)}_i $. Let $\wt{S}_0^{[2]}\dashrightarrow X$ be the flop of $R^{(2)}_1,\ldots,R^{(2)}_{k}$ i.e.~the blow-up of each $R^{(2)}_i\cong\PP^2$ followed by contraction of the exceptional fiber $E_i$ (which is isomorphic to the incidence variety in $\PP^2\times(\PP^2)^{\vee}$) along the projection $E_i \to (\PP^2)^{\vee}$. Map~\eqref{corda} extends~\cite{ferretti} to a regular degree-$6$ map
\begin{equation}\label{estendo}
X \lra  Q(U).
\end{equation}
The following result is due to Ferretti:
\begin{prp}[Ferretti, Prop.~4.3 of~\cite{ferretti}]\label{prp:degequar}
Keep notation as above. There exist a commutative diagram 
\begin{equation}\label{trekking}
\xymatrix{  
\cX \ar_{\pi}[dr]  \ar^{G}[rr] &      &  \cU\times\PP^5 \ar[dl] \\
  & \cU &}
\end{equation}
and  maps
\begin{equation*}
\cU\overset{A}{\lra}\lagr,\qquad \cU\overset{\Lambda_i}{\lra}\Gr(2,\PP^5),\quad i=1,\ldots,k
\end{equation*}
such that the following hold:
\begin{enumerate}
\item[(1)]
$\cU$ is a connected contractible manifold of dimension $(20-k)$.
\item[(2)]
$\pi$ is a proper map and a submersion of complex manifolds: for $t\in \cU$ we let $X_t:=\pi^{-1}(t)$ and $g_t\colon X_t\to\PP^5$ be the regular map induced by $G$.
\item[(3)]
There exists $0\in\cU$ and an isomorphism  $X_0\cong X$ such that $g_0$ gets identified with Map~\eqref{estendo}. Moreover $A(0)=A_{+}(U)$ and $\Lambda_i(0)=i_{+}(p_i)$.
\item[(4)]
There exist   a regular map $c_t\colon X_t\to X_{A(t)}$ and  prime divisors $E_i(t)$ on $X_t$   for $i=1,\ldots,k$  such that
 the following hold  for all $t$ belonging to an open dense $\cU^0\subset\cU$:
\begin{enumerate}
\item[(4a)]
$g_t=f_{A(t)}\circ c_t$.
\item[(4b)]
$g_t(E_i(t))=\Lambda_i(t)$ for $i=1,\ldots,k$.
\item[(4c)]
$c_t$  contracts  each $E_i(t)$ to a  
 $K3$ surface $S_i(t)\subset X_{A(t)}$
and is an isomorphism of the complement of $\cup_{i=1}^{k}E_i(t)$ onto its image.  
\end{enumerate}
\item[(5)]
The period map $\cU\to\PP(H^2(X_0;\CC))$ is an immersion i.e.~the family of deformations of $X_0$ parametrized by $\cU$ has $(20-k)$ moduli.
\end{enumerate}
\end{prp}  
 Given~\Ref{prp}{degequar} it is easy to show that there exist  complete families of pairwise incident planes of cardinality $k$ for $10\le k\le 16$. Before stating the relevant result we recall that the $K3$ surface  $\wt{S}_0$ depends on the choice of nodes ${\bf p}_{i_1},\ldots,{\bf p}_{i_k}$ and hence so does the variety $X$. 
 \begin{prp}
Keep notation as in~\Ref{prp}{degequar} and let $10\le k\le 16$. Let $t\in\cU^0$ be close to $0$. One can choose the nodes ${\bf p}_{i_1},\ldots,{\bf p}_{i_k}$ of ${\bf K}$ so that  ${\bf \Theta}_{A(t)}$ is a complete family of pairwise incident planes of cardinality $k$. 
\end{prp}
 \begin{proof}
 The map $i_{+}$ is identified with the map associated to the complete  linear system $|\cO_{\PP(U)}(2)|$. It  is well-know     that no quadric in $\PP(U)$ contains ${\bf p}_1,\ldots,{\bf p}_{16}$\footnote{Suppose that the quadric $Q_0$ contains $p_1,\ldots,p_{16}$. There exist $16$ planes $L_1,\ldots,L_{16}\subset\PP(U)$ such that each $L_j$ contains $6$ of the nodes of ${\bf K}$ and moreover there is a unique smooth conic $C_j\subset L_j$ containing the six nodes. It follows that  $Q_0$ contains $C_1,\ldots,C_{16}$ and hence $Q_0\cap {\bf K}$ has degree at least $32$:  that contradicts B\'ezout.}. Since $10\le k\le 16 $ we may choose ${\bf p}_{i_1},\ldots,{\bf p}_{i_k}$ such  that no quadric in $\PP(U)$ contains them. Thus  $i_{+}({\bf p}_{i_1}),\ldots,i_{+}({\bf p}_{i_k})$ span a $10$-dimensional subspace of $\PP(\bigwedge^3\CC^6)$. It follows that for small enough $t\in\cU^0$ the planes $\Lambda_1,\ldots,\Lambda_k$ span $\PP(A(t))$.
 By~\Ref{clm}{incilagr}   it remains to prove that no other plane is contained in ${\bf \Theta}_{A(t)}$. Suppose that $\Lambda\in{\bf \Theta}_{A(t)}$ and that $\Lambda\notin\{\Lambda_1(t),\ldots,\Lambda_{k}(t)\}$. By Item~(2) of~\Ref{prp}{neve} we get that $X_{A(t)}$ is singular along $f_{A(t)}^{-1}(\Lambda)$: since  $\Lambda\notin\{\Lambda_1(t),\ldots,\Lambda_{k}(t)\}$ that contradicts Item~(4c) of~\Ref{prp}{degequar}.
\end{proof}
 \section{Upper bound}\label{sec:stimasup}
 \setcounter{equation}{0}
 We will prove that a finite complete family of pairwise incident planes in $\PP^5$ has at most $20$ elements. The key element in the proof is the following construction from~\cite{ogtanto}: given $A\in\lagr$ and $W\in\Theta_A$ we consider the locus
 \begin{equation}
C_{W,A}:=\{[v]\in\PP(W) \mid \dim(F_v\cap A)\ge 2\}.
\end{equation}
(Notice that $\dim(F_v\cap A)\ge 1$ for $[v]\in \PP(W)$ because $\bigwedge^3 W\subset(F_v\cap A)$.) One describes $C_{W,A}$ as the degeneracy locus of a map between vector-bundles of rank $9$: the fiber over $[v]$ of the domain  is equal to  $F_v/\bigwedge^3 W$, the codomain  is the trivial vector-bundle with fiber $\bigwedge^3 W^{\bot}/\bigwedge^3 W$ - see~\cite{ogtanto} for details. It follows that either $C_{W,A}=\PP(W)$ or else $C_{W,A}$ is a sextic curve. The link with our problem is the following. Suppose that $C_{W,A}\not=\PP(W)$ and that $W'\in\Theta_A$ is distinct from $W$:   then $\PP(W\cap W')$ is contained in the singular locus of $C_{W,A}$.  In order to state the relevant results from~\cite{ogtanto} we give a couple of definitions. Let $W\subset V$ be a subspace: we let
\begin{equation}\label{essewu}
S_W:= (\bigwedge^2 W)\wedge \CC^6.
\end{equation}
\begin{dfn}\label{dfn:malvagio}
Let $A\in\lagr$ and suppose that $W\in\Theta_A$. We let $\cB(W,A)\subset\PP(W)$ be the set of $[v]$ such that one of the following holds:
\begin{itemize}
\item[(1)]
There exists $W'\in(\Theta_A\setminus\{W\})$ such that  $[v]\in \PP(W')$.
\item[(2)]
$\dim(A\cap F_{v}\cap S_W)\ge 2$.
\end{itemize}
\end{dfn}
One checks easily that $\cB(W,A)$ is closed subset of $\PP(W)$.
\begin{prp}[\rm Corollary 3.2.7 of~\cite{ogtanto}]\label{prp:cnesinerre}
Let $A\in\lagr$ and suppose that $W\in\Theta_A$. Then $C_{W,A}=\PP(W)$ if and only if 
$\cB(W,A)=\PP(W)$.  If $C_{W,A}\not=\PP(W)$ then  $\cB(W,A)\subset \sing C_{W,A}$.
\end{prp}
\begin{lmm}\label{lmm:unridotto}
Let $A\in\lagr$. Suppose that $\Theta_A$ is finite of cardinality at least $15$. Then there exists $W\in\Theta_A$ such that $C_{W,A}$ is a reduced curve.
\end{lmm}
\begin{proof}
By contradiction. Assume that for every $W\in\Theta_A$  one of the following holds:
\begin{enumerate}
\item[(1)]
$C_{W,A}=\PP(W)$.
\item[(2)]
$C_{W,A}$ is a non-reduced curve.
\end{enumerate}
By~\Ref{prp}{cnesinerre} we get that $\dim\cB(W,A)\ge 1$. Let $W'\in(\Theta_A\setminus\{W\})$: since $\Theta_A$ is finite the planes $\PP(W)$ and $\PP(W')$ intersect in a single point, see~\Ref{crl}{unpunto}. It follows that for generic $[v]\in\cB(W,A)$ there exists 
\begin{equation}\label{rugby}
\alpha\in\left((A\cap F_{v}\cap S_W)\setminus\bigwedge^3 W\right). 
\end{equation}
Given such $\alpha$ there is a unique $[v]\in\PP(W)$ such that~\eqref{rugby} holds. In fact suppose the contrary: then $\alpha$ is a decomposable element whose support is a $W'\in(\Theta_A\setminus\{W\})$ intersecting $W$ in a $2$-dimensional subspace, that contradicts the hypothesis that $\Theta_A$ is finite (see above). 
Since  $\dim\cB(W,A)\ge 1$ it follows that 
\begin{equation}\label{football}
\dim(A\cap  S_W)\ge 3. 
\end{equation}
Thus $\PP(A)$ intersects the projective tangent space to $\Gr(2,\PP^5)$ (embedded by Pl\"ucker) at $\PP(W)$ in a linear space of dimension at least $2$. Now let $\Omega\subset\PP(\bigwedge^3\CC^6)$ be a generic $10$-dimensional projective space containing $\PP(A)$.  Notice that 
\begin{equation*}
\dim\Omega+\dim\Gr(2,\PP^5)=19=\dim\PP(\bigwedge^3\CC^6). 
\end{equation*}
The intersection $\Omega\cap\Gr(2,\PP^5)$ is finite because by hypothesis ${\bf \Theta}_A=\PP(A)\cap \Gr(2,\PP^5)$ is finite.
By~\eqref{football} we get that $\Omega$ intersects the projective tangent space to $\Gr(2,\PP^5)$  at $\PP(W)$ in a linear space of dimension at least $2$: thus 
\begin{equation}
\mult_{\PP(W)}\Omega\cdot \Gr(2,\PP^5)\ge 3.
\end{equation}
Since the cardinality of ${\bf \Theta}_A$ is at least $15$ we get that $\Omega\cdot \Gr(2,\PP^5)\ge 45$, that is a contradiction because   $\deg\Gr(2,\PP^5)=42$, see p.~247 of~\cite{joebook}.
\end{proof}
Now let $T$ be a finite complete family of pairwise incident planes in $\PP^5$. By~\Ref{clm}{incilagr} there exists $A\in\lagr$ such that ${\bf\Theta}_A=T$.  Suppose that $T$ has cardinality at least $15$: by~\Ref{lmm}{unridotto} there exists $W\in \Theta_A$ such that $C_{W,A}$ is a reduced sextic curves.   Let $W'\in(\Theta_A\setminus\{W\})$: by~\Ref{crl}{unpunto} the intersection $\PP(W)\cap \PP(W')$ is a point. By~\Ref{prp}{cnesinerre} the curve $C_{W,A}$ is singular at $\PP(W)\cap \PP(W')$. Thus we have a map
\begin{equation}\label{interseco}
\begin{matrix}
\Theta_A\setminus\{W\} & \overset{\varphi}{\lra} & \sing C_{W,A} \\
W' & \mapsto &  \PP(W')\cap \PP(W)
\end{matrix}
\end{equation}
There are at most $15$ singular points of $C_{W,A}$ (the maximum $15$ is achieved by sextics which are the union of $6$ generic lines): it follows that if $\varphi$ is injective then ${\bf \Theta}_A=T$ has at most $16$ elements. Since  $\varphi$ is not necessarily injective we will need to answer the following question: what is the relation between the cardinality of  $\varphi^{-1}(p)$ and the singularity of $C_{W,A}$ at $p$ ? First we will recall how to compute the initial terms in the Taylor expansion of a local equation of $C_{W,A}$ at a given point $[v_0]\in\PP(W)$ - here $A\in\lagr$ and $W\in\Theta_A$ are arbitrary.   Let $[w]\in\PP(W)$;   we let 
\begin{equation}
G_w:=F_w/\bigwedge^ 3 W.
\end{equation}
Let $W_0\subset W$  be a subspace complementary to $[v_0]$. We have an isomorphism
\begin{equation}\label{minnie}
\begin{matrix}
W_0 & \overset{\sim}{\lra} & \PP(W)\setminus\PP(W_0) \\
w & \mapsto & [v_0 +w]
\end{matrix}
\end{equation}
onto a neighborhood of $[v_0]$;
thus $0\in W_0$ is identified with $[v_0]$. We have
\begin{equation}\label{qumme}
C_{W,A}\cap W_0=V(g_0+g_1+\cdots+g_6),\qquad g_i\in \Sym^i W_0^{\vee}.
\end{equation}
Given  $w\in W$ we define the Pl\"ucker quadratic form $\psi^{v_0}_w$ on $G_{v_0}$ as follows. Let $\ov{\alpha}\in G_{v_0}$ be the equivalence class of $\alpha\in F_{v_0}$. Thus $\alpha=v_0\wedge\beta$ where $\beta\in\bigwedge^ 2 V$  is defined modulo $(\bigwedge^ 2 W+ [v_0]\wedge V)$: we let 
\begin{equation}\label{quadriwu}
\psi^{v_0}_w(\ov{\alpha}):=
\vol(v_0\wedge w\wedge\beta\wedge\beta).
\end{equation}
\begin{prp}[Prop.~3.1.2 of~\cite{ogtanto}]\label{prp:primisarto}
Keep notation and hypotheses as above. Let  $\ov{K}:=A \cap F_{v_0}/\bigwedge^3 W$ (notice that $\ov{K}\subset G_{v_0}$) and $\ov{k}:=\dim \ov{K}=\dim(A\cap F_{v_0})-1$.   Then the following hold:
\begin{itemize}
\item[(1)]
$g_i=0$ for $i<\ov{k}$.
\item[(2)]
 There exists $\mu\in\CC^{*}$ such that
\begin{equation}\label{marzamemi}
g_{\ov{k}}(w)=\mu\det(\psi^{v_0}_w|_{\ov{K}}),\quad w\in W_0.
\end{equation}
\end{itemize}
\end{prp} 
Next we will give a geometric interpretation of the right-hand side of~\eqref{marzamemi}.  Choose a  subspace $V_0\subset\CC^6$ complementary to $[v_0]$ and  such that 
 $V_0\cap W=W_0$.
Thus have  isomorphisms
\begin{equation}\label{canemaschio}
\begin{matrix}
\bigwedge^ 2 V_0 & \overset{\sim}{\lra} & F_{v_0} \\
\beta & \mapsto & v_0\wedge\beta
\end{matrix}
\end{equation}
and
\begin{equation}\label{giquoz}
\begin{matrix}
\bigwedge^ 2 V_0/\bigwedge^ 2 W_0 & \overset{\sim}{\lra} & G_{v_0} \\
\overline{\beta} & \mapsto & \overline{v_0\wedge\beta}.
\end{matrix}
\end{equation}
Let  $\psi^{v_0}_w$ be as in~\eqref{quadriwu}: we will view it as a quadratic form on $\bigwedge^ 2 V_0/\bigwedge^ 2 W_0$ via~\eqref{giquoz}. Let $V(\psi^{v_0}_w)\subset\PP(\bigwedge^ 2 V_0/\bigwedge^ 2 W_0)$ be the zero-locus of $\psi^{v_0}_w$. Let 
\begin{equation}\label{vaporub}
\wt{\rho}\colon\PP(\bigwedge^ 2 V_0)\dashrightarrow
\PP(\bigwedge^ 2 V_0/\bigwedge^ 2 W_0)
\end{equation}
be projection with center $\bigwedge^ 2 W_0$. 
Let  
\begin{equation}\label{prograss}
{\mathbb Gr}(2,V_0)_{W_0}:=
\wt{\rho}({\mathbb Gr}(2,V_0)).
\end{equation}
(The right-hand side is to be interpreted as the 
 closure of $\wt{\rho}({\mathbb Gr}(2,V_0)\setminus\{\bigwedge^ 2 W_0\})$.) Let $\rho$ be the restriction of $\wt{\rho}$ to ${\mathbb Gr}(2,V_0)$. The rational map
\begin{equation}\label{proraz}
\rho\colon{\mathbb Gr}(2,V_0)\dashrightarrow
{\mathbb Gr}(2,V_0)_{W_0}
\end{equation}
is birational because ${\mathbb Gr}(2,V_0)$ is cut out by quadrics. The following is an easy exercise, see Claim~3.5 of~\cite{ogtanto}.
\begin{clm}\label{clm:roncisvalle}
Keep notation as above. Then 
\begin{equation}\label{grascop}
\bigcap_{w\in W_0}V(\psi^{v_0}_w)={\mathbb Gr}(2,V_0)_{W_0}
\end{equation}
and the scheme-theoretic intersection on the left is reduced.
\end{clm}
Let $A\in\lagr$ and suppose that $W\in\Theta_A$.  Let $p\in\PP(W)$. We let
\begin{equation}
n_p:=\#\{W'\in(\Theta_A\setminus\{W\}) \mid p\in\PP(W')\}.
\end{equation}
Notice that if $n_p>0$ then $p\in C_{W,A}$.
\begin{prp}\label{prp:cardmult}
Let $A\in\lagr$ and suppose that $\Theta_A$ is finite. Assume that $W\in\Theta_A$. Let $p\in\PP(W)$. 
\begin{enumerate}
\item
$n_p\le 4$.
\item
Assume in addition that $C_{W,A}$ is a curve. Then the following hold:
\begin{enumerate}
\item[(2a)] 
If $n_p=2$ then either $C_{W,A}$ has a cusp\footnote{By  cusp we mean a plane curve singularity with tangent cone which is quadratic of rank $1$.} at $p$ or else $\mult_p C_{W,A}\ge 3$.
\item[(2b)]
If $n_p=3$ or  $n_p=4$ then  $\mult_p C_{W,A}\ge 3$
\end{enumerate}
\end{enumerate}
\end{prp}
\begin{proof}
Throughout the proof we will let $p=[v_0]$. Let $K:=A\cap F_{v_0}$: we will view $K$ as a subspace of $\bigwedge^2 V_0$ via Isomorphism~\eqref{canemaschio}.  (1): Suppose that $n_p>4$. 
We claim that $\dim K\ge 4$. In fact suppose that $\dim K\le 3$ i.e.~$\dim\PP(K)\le 2$. Since $n_p\ge 5$ the intersection $\PP(K)\cap\Gr(2,V_0)$ contains at least $6$ points: that is absurd because $\Gr(2,V_0)$ is cut out by quadrics and the intersection $\PP(K)\cap\Gr(2,V_0)$ is finite (recall that $\Theta_A$ is finite by hypothesis). This proves that $\dim K\ge 4$. Since $\PP(K)\cap\Gr(2,V_0)$ is finite we get that $\dim\PP(K)\le 3$ and hence $\dim\PP(K)=3$. Since the degree of $\Gr(2,V_0)$ is $5$ and $\PP(K)\cap\Gr(2,V_0)$ contains at least $6$ points we get that $\PP(K)\cap\Gr(2,V_0)$ is infinite: that is a contradiction. (2a): If $\dim K\ge 4$ then  $\mult_p C_{W,A}\ge 3$ by Item~(1) of~\Ref{prp}{primisarto}. Suppose that $\dim K< 4$ i.e.~$\dim\PP(K)\le 2$. By hypothesis $\PP(K)\cap\Gr(2,V_0)$ is finite and contains $3$ points. Since $\Gr(2,V_0)$  is cut out by quadrics we get that $\dim\PP(K)=2$. Let $g_0,\ldots,g_6$ be as in~\eqref{qumme}. Then $0=g_0=g_1$ because $\dim K=3$ (see Item~(1) of~\Ref{prp}{primisarto}) and $g_2$ is given by~\eqref{marzamemi}. Let $\wt{\rho}$ be the projection of~\eqref{vaporub}. The closure of $\wt{\rho}(\PP(K)\setminus\bigwedge^2 W_0)$ is a line intersecting $\Gr(2,V_0)_{W_0}$ in two distinct points, namely the images under projection of the two points belonging to  $(\PP(K)\setminus\bigwedge^2 W_0)\cap \Gr(2,V_0)$. By~\eqref{qumme} and~\Ref{clm}{roncisvalle} we get that $g_2=l^2$ where $0\not=l\in W_0^{\vee}$: thus $C_{W,A}$ has a cusp at $p$. (2b): We will prove that $\dim K\ge 4$ - then $\mult_p C_{W,A}\ge 3$ will follow from Item~(1) of~\Ref{prp}{primisarto}. Assume that $\dim K< 4$. Suppose that $n_p=3$. Then $\PP(K)\cap \Gr(2,V_0)$ has cardinality $4$. Since $\Gr(2,V_0)$ is cut out by quadrics we get that $\dim\PP(K)=2$ and no three among  the points of  $\PP(K)\cap \Gr(2,V_0)$ are collinear.
 Now project $\PP(K)$ from $\bigwedge^2 W_0$ - see~\eqref{vaporub}: we get that $\wt{\rho}(\PP(K)\setminus\bigwedge^2 W_0)$ is a line intersecting $\Gr(2,V_0)_{W_0}$ in three distinct points, that contradicts~~\Ref{clm}{roncisvalle}. We have proved that   if $n_p=3$ then $\mult_p C_{W,A}\ge 3$. Lastly suppose that $n_p=4$. Then  $\PP(K)\cap \Gr(2,V_0)$ has cardinality $5$ and $\dim\PP(K)\le 2$: that is absurd because $\Gr(2,V_0)$ is cut out by quadrics.
\end{proof}
Now let $A\in\lagr$ and assume that $\Theta_A$ is finite of cardinality at least $15$. By~\Ref{lmm}{unridotto} there exists $W\in\Theta_A$ such that $C_{W,A}$ is a reduced curve. We let 
\begin{equation}\label{trota}
L_j:=\{p\in\PP(W) \mid n_p=j\},\qquad \ell_j:=\# L_j. 
\end{equation}
By~\Ref{prp}{cardmult} we have that $\ell_j=0$ for $j>4$ and hence 
\begin{equation}\label{numpiani}
\#\Theta_A\le 1+\ell_1+2\ell_2+3\ell_3+4\ell_4.
\end{equation}
\begin{lmm}\label{lmm:pasquetta}
Let $A\in\lagr$ and assume that $\Theta_A$ is finite of cardinality at least $15$.  Let $W\in\Theta_A$ be  such that $C_{W,A}$ is a reduced curve and keep notation as above.
Let  $s$ be the number of irreducible components of $C_{W,A}$.
Then 
\begin{equation}\label{stimasup}
\ell_1+\ell_2+3\ell_3+3\ell_4\le 9+s.
\end{equation}
\end{lmm}
\begin{proof}
Let $C:=C_{W,A}$ and $\mu\colon Z\to\PP^2$ be a series of blow-ups that desingularize $C$ i.e.~such that the strict transform $\wt{C}\subset Z$ is smooth. Then 
\begin{equation}\label{aggiunzione}
-2s\le 2(h^0(K_{\wt{C}})-h^1(K_{\wt{C}}))= 2\chi(K_{\wt{C}})=\wt{C}\cdot\wt{C}+\wt{C}\cdot K_Z.
\end{equation}
On the other hand let $p\in\PP(W)$: if $n_p\ge 1$  then $C$ is singular at $p$ and if $n_p\ge 3$ then the multiplicity of $C$ at $p$ is at least $3$, see~\Ref{prp}{cardmult}. 
It follows that
\begin{equation}\label{scoppio}
\wt{C}\cdot\wt{C}+\wt{C}\cdot K_Z\le (C\cdot C+C\cdot K_{\PP^2})-2(\ell_1+\ell_2)-6(\ell_3+\ell_4)=
18-2(\ell_1+\ell_2)-6(\ell_3+\ell_4).
\end{equation}
The proposition follows from~\eqref{aggiunzione} and~\eqref{scoppio}.
\end{proof}
The result below completes the proof of~\Ref{thm}{alpiuventi}.
\begin{prp}\label{prp:casopercaso}
Let $A\in\lagr$ and assume that $\Theta_A$ is finite.   Then
\begin{equation}
\#\Theta_A\le 20
\end{equation}
\end{prp}
\begin{proof}
We may assume that $\#\Theta_A>16$. By~\Ref{lmm}{unridotto} there exists $W\in\Theta_A$ such that $C_{W,A}$ is a reduced curve. Let  $L_j$ and $\ell_j$ be as in~\eqref{trota} and $s$ be the number of irreducible components of $C_{W,A}$.  We recall that $C_{W,A}$ is singular at each point of $L_1$, it has either a cusp or a point of multiplicity at least $3$ at each point of $L_2$ and it has multiplicity at least $3$ at each point of $L_3\cup L_4$, see~\Ref{prp}{cardmult}. 
By~\eqref{numpiani} we have that
\begin{equation}
16\le \ell_1+2\ell_2+3\ell_3+4\ell_4
\end{equation}
   The proof consists of a case-by-case analysis. Suppose that $s=1$.  
   Assume that $(\ell_3+\ell_4)=0$.  Applying Pl\"ucker's formulae to $C_{W,A}$ we get that $(2\ell_1+3\ell_2)\le 27$: it follows that  
   $\#\Theta_A\le 19$ (recall~\eqref{numpiani}). Assume that $(\ell_3+\ell_4)=1$. Then $(\ell_1+\ell_2)\le 7$ by~\Ref{lmm}{pasquetta}: it follows that $\#\Theta_A\le 19$.    If $(\ell_3+\ell_4)=2$ then $(\ell_1+\ell_2)\le 4$ by~\Ref{lmm}{pasquetta}: it follows that $\#\Theta_A=17$. Suppose that $s=2$.  A similar analysis shows that necessarily\footnote{Notice that if an irreducible plane quintic has $5$ cusps then it is smooth elsewhere (project the quintic form a hypothetical singular point distinct from the $5$ cusps: you will contradict Hurwitz' formula).}  $C_{W,A}=D+L$ where $D$  is an irreducible quintic with $4$ cusps (the points of $L_2$) and $2$ nodes (the points of $L_4$), $L$ is the line through the nodes of $D$: thus
   $\#\Theta_A= 17$. If $s\ge 3$ then $C_{W,A}=D_1+D_2+D_3$ where $D_1$, $D_2$ and $D_3$ are reduced conics (eventually reducible) belonging to the same pencil with reduced base locus (which is equal to $L_3\cup L_4$). We have $\#\Theta_A\le (17+\delta)$ where $\delta$ is the number of singular conics among $\{D_1,D_2,D_3\}$.
\end{proof}
\end{document}